\documentclass[11pt, oneside]{article}   	% use "amsart" instead of "article" for AMSLaTeX format
\usepackage{geometry}                		% See geometry.pdf to learn the layout options. There are lots.
\usepackage{graphicx}				% Use pdf, png, jpg, or epsÂ§ with pdflatex; use eps in DVI mode
\usepackage[usenames]{xcolor}		
\usepackage{amsmath,amsthm,amssymb}
\usepackage{amscd}

\theoremstyle{definition}
\newtheorem{dfn}{Definition}[section]

\theoremstyle{plain}
\newtheorem{thm}[dfn]{Theorem}

\newtheorem{lem}[dfn]{Lemma}
\newtheorem{cor}[dfn]{Corollary}

\newcommand{\Z}{\mathbb{Z}}

\newcommand{\aut}{\mathrm{Aut}}

\newcommand{\rot}{\mathrm{rot}}

\newcommand{\gArf}{\widehat{\operatorname{Arf}}}

\newcommand{\image}{\operatorname{im}}

\title{The mapping class group orbits in the framings of compact surfaces
}
\author{Nariya Kawazumi}

\begin{document}
\maketitle

\begin{abstract}
We compute the mapping class group orbits in the homotopy set 
of framings of a compact connected oriented surface with non-empty boundary.
In the case $g \geq 2$ the computation is some modification 
of Johnson's results \cite{J80a}\cite{J80b} and certain arguments
on the Arf invariant, 
while we need an extra invariant for the genus $1$ case.
In addition, we discuss how this invariant behaves in the relative case, 
which Randal-Williams \cite{RW14} studied for $g \geq 2$. 
\end{abstract}

\begin{center}
Introduction
\end{center}

Let $\Sigma$ be a compact connected oriented smooth ($C^\infty$) 
surface with non-empty boundary. Then the tangent bundle 
$T\Sigma$ is a trivial bundle. Its orientation-preserving 
global trivializations $T\Sigma \overset\cong\to 
\Sigma\times\mathbb{R}^2$ are called 
{\it framings} of the surface $\Sigma$, which play important roles 
in surface topology.  The $\bmod 2$ reduction of a framing can be regarded 
as a spin structure on the surface $\Sigma$. A spin structure on a closed surface
is called a theta characterisic in a classical context, and the mapping class group 
orbits in the set of theta characteristics are described by the Arf invariant
\cite{A41}. \par
We denote by $F(\Sigma)$ the set of 
homotopy classes of framings of $\Sigma$, and fix a Riemannian metric 
$\Vert\cdot\Vert$ on the tangent bundle $\varpi: T\Sigma \to \Sigma$. 
The unit tangent bundle $U\Sigma := \{e \in T\Sigma;\,\, \Vert e\Vert = 1\}
\overset\varpi\to \Sigma$ is a principal $S^1$ bundle over $\Sigma$. 
A framing defines a continuous map $U\Sigma \to S^1$ whose restriction 
to each fiber is homotopic to the identity $1_{S^1}$.
Taking the pull-back of the positive generator of $H^1(S^1; \Z)$, 
we obtain an element of $H^1(U\Sigma; \Z)$. This defines a natural 
embedding $F(\Sigma)\hookrightarrow H^1(U\Sigma; \Z)$.
More precisely, $F(\Sigma)$ is an affine set modeled by the abelian group
$\varpi^*H^1(\Sigma; \Z) (\cong H^1(\Sigma; \Z))$ (See \S2.1).
In particular, the difference $f_1 - f_0$ of two framings $f_0$ and 
$f_1 \in F(\Sigma)$ defines a unique element of $H^1(\Sigma; \Z)$.
\par
In this paper we consider the mapping class group of $\Sigma$ fixing 
the boundary {\it pointwise}
$$
\mathcal{M}(\Sigma) := 
\pi_0\operatorname{Diff}_+
(\Sigma, \mathrm{id}\, \mathrm{on} \,\partial\Sigma)
= \operatorname{Diff}_+
(\Sigma, \mathrm{id}\, \mathrm{on} \,\partial\Sigma)/
\text{isotopy},
$$
which acts on the set $F(\Sigma)$ from the right in a natural way.
If we fix an element $f_0 \in F(\Sigma)$, then the map
$$
k(f_0): \mathcal{M}(\Sigma) \to H^1(\Sigma; \Z), \quad
\varphi \mapsto f_0\circ\varphi - f_0, 
$$
is a twisted cocycle of the group $\mathcal{M}(\Sigma)$.
The cohomology class $k:=[k(f_0)] \in H^1(\mathcal{M}(\Sigma); \linebreak
H^1(\Sigma; \Z))$ does not depend on the choice of $f_0$, 
is called the Earle class \cite{E78} or 
the Chillingworth class \cite{C72a} \cite{C72b} \cite{T92}, and 
generates the cohomology group in the case when
the boundary $\partial\Sigma$ is connected and the genus of $\Sigma$ 
is greater than $1$ \cite{Mo89}.
For the case where the boundary is not connected, see \cite{Ka08}
Theorem 1.A. 
The construction of $k$ stated here is due to M.\ Furuta \cite{Mo97} \S4.
The Morita trace \cite{Mo93} and its refinement, the Enomoto-Satoh trace
\cite{ES14}, are higher analogues of the class $k$. In the author's 
joint paper with Alekseev, Kuno and Naef \cite{AKKN}, we clarify
topological and Lie theoretical meanings of the Enomoto-Satoh trace.
The formality problem of a variant of the Turaev cobracket for an immersed loop 
on the surface, the Enomoto-Satoh trace and the Kashiwara-Vergne problem
in Lie theory are closely related to each other. 
We need the rotation number of the immersed loop with respect to a 
framing to define of this variant of the Turaev cobracket.
This is the reason why we describe the orbit set 
$F(\Sigma)/\mathcal{M}(\Sigma)$ in this paper.\par

The homotopy set $F(\Sigma)$ we study in this paper is 
\textit{absolute}, namely, we allow framings to move on the boundary.
In fact, the rotation number of an immersed loop with respect to a framing $f$
is invariant under any moves of $f$ on the boundary $\partial\Sigma$.
On the other hand, we can consider a \textit{relative} version of the homotopy 
set $F(\Sigma, \delta)$ for a fixed framing on the boundary
$\delta: TS\vert_{\partial\Sigma}\overset\cong\to 
\partial\Sigma\times\mathbb{R}^2$. Here we make framings on $\partial \Sigma$
equal the given datum $\delta$. We need the latter version to define 
the rotation number of an arc connecting two boundary components.
Randal-Williams \cite{RW14} computes the mapping class group orbits 
in the set of ($r$-)spin structures for any genus in the relative version 
and those in the homotopy set $F(\Sigma, \delta)$ for $g\geq 2$.
It is interesting that the (generalized) Arf invariant is defined in any 
$F(\Sigma, \delta)$ \cite{RW14}, while it is not defined in some absolute 
cases as in \S1 of this paper. In particular, the computations in this paper
are different from those by Randal-Williams \cite{RW14}. 
In the case $g \geq 2$, the formality of the Turaev cobracket holds good 
for any choice of a framing. But, if $g=1$, it depends on the choice of 
a framing, so that the formality problem is reduced to the computation 
of the mapping class group orbits in the set $F(\Sigma)$.
It is controlled by an extra invariant $\tilde A(f)$ introduced in this paper
(Corollary \ref{cor:realize1}). All these results are proved in \cite{AKKN17b}.
\par
Anyway, following Whitney \cite{W37}, we consider 
the rotation number $\rot_f(\ell) \in \mathbb{Z}$ of 
a smooth immersion $\ell: S^1 \to \Sigma$ with 
respect to a framing $f \in F(\Sigma)$. 
We number the boundary components as $\partial\Sigma
= \coprod^n_{j=0}\partial_j\Sigma$. 
The rotation numbers $\rot_f(\partial_j\Sigma)$, $0 \leq j\leq n$, 
are invariant under the action of the group $\mathcal{M}(\Sigma)$.
Here we endow each $\partial_j\Sigma$ with the orientation induced 
by $\Sigma$. By the Poincar\'e-Hopf theorem (Lemma \ref{lem:PH}), we have
$$
\sum^n_{j=1}\rot_f(\partial_j\Sigma) = \chi(\Sigma) = 1-2g-n.
$$
\par\smallskip

Our description of the orbit set $F(\Sigma)/\mathcal{M}(\Sigma)$ 
depends on the genus $g(\Sigma)$ of the surface $\Sigma$.
%\par
First we consider the case $g(\Sigma)=0$. Clearly we have
\begin{lem}[Equation \eqref{eq:genus0}]
 Suppose $g(\Sigma)=0$. Then two framings 
$f_1$ and $f_2 \in F(\Sigma)$ are homotopic 
to each other, if and only if
$$
\rot_{f_1}(\partial_j\Sigma) = \rot_{f_2}(\partial_j\Sigma)
$$
for any $0 \leq j \leq n$.
\end{lem}
\par\smallskip
Next we discuss the positive genus case: $g = g(\Sigma) \geq 1$. 
Choose a system of simple closed curves $\{\alpha_i, 
\beta_i\}^g_{i=1}$ on $\Sigma$ as in Figure 1. 
\begin{figure}
\begin{center}
\input{fig1-80percent.tex}
\end{center}
\label{fig1}
\caption{}
\end{figure}
%Namely
%$\alpha_i$ and $\beta_i$ are disjoint from $\{\alpha_k, 
%\beta_k\}_{k\neq i}$, and intersect at one point in a transverse 
%and positive way, and the set $\{\alpha_i, 
%\beta_i\}^g_{i=1}\cup \{\partial_j\Sigma\}^n_{j=1}$ constitutes
%a free basis of $H_1(\Sigma_{g,n+1}; \mathbb{Z})$. 
The Arf invariant of the $\bmod 2$ reduction of $f$ is defined 
in the case where 
all the numbers $\rot_f(\partial_j\Sigma)$, $0 \leq j \leq n$, are odd.
Then the Arf invariant 
of the spin structure is defined by 
$$
\operatorname{Arf}(f) \equiv \sum^g_{i=1}
(\rot_f(\alpha_i)+1)(\rot_f(\beta_i)+1) \pmod{2}.
$$
\par\smallskip
In the case $g(\Sigma)\geq 2$, we have the following.
\begin{thm}[Theorem \ref{thm:genus2}] 
Suppose $g(\Sigma) \geq 2$, and $f_1, f_2 \in F(\Sigma)$.
Then $f_1$ and $f_2$ belong to the same 
$\mathcal{M}(\Sigma)$-orbit, if and only if
\begin{enumerate}
\item[(i)] $
\rot_{f_1}(\partial_j\Sigma) = \rot_{f_2}(\partial_j\Sigma)
$
for any $0 \leq j \leq n$.
\item[(ii)] If all the numbers 
$\rot_{f_1}(\partial_j\Sigma) = \rot_{f_2}(\partial_j\Sigma)$,
$0 \leq j \leq n$, 
are odd, then $\operatorname{Arf}(f_1) = \operatorname{Arf}(f_2)$.
\end{enumerate}
\end{thm}
The proof given in \S2.2 is some modification of 
Johnson's arguments \cite{J80a} \cite{J80b}. 
\par\smallskip
The genus $1$ case is different from the others. 
We need to introduce an invariant $\tilde A(f) \in \mathbb{Z}_{\geq0}$
for $f \in F(\Sigma)$. It is defined to be the generator 
of the ideal in $\mathbb{Z}$ generated by the set $\{\rot_f(\gamma); 
\,\,\text{$\gamma$ is a non-separating simple closed curve on $\Sigma$}\}$.
We have 
$$
\operatorname{Arf}(f) \equiv \tilde A(f) + 1 \pmod{2}.
$$
On the other hand, if $g \geq 2$, we have $\tilde A(f) = 1$ for any $f \in 
F(\Sigma)$ (Lemma \ref{lem:genA}).
\begin{thm}[Theorem \ref{genusone}]
Suppose $g(\Sigma) = 1$, and $f_1, f_2 \in F(\Sigma)$.
Then $f_1$ and $f_2$ belong to the same 
$\mathcal{M}(\Sigma)$-orbit, if and only if
\begin{enumerate}
\item[(i)] $
\rot_{f_1}(\partial_j\Sigma) = \rot_{f_2}(\partial_j\Sigma)
$
for any $0 \leq j \leq n$.
\item[(ii)] $\tilde A(f_1) = \tilde A(f_2)\in \mathbb{Z}_{\geq0}$.
\end{enumerate}
\end{thm}

%\bigskip
For the sake of non-experts on topology who are interested only in the 
Kashiwara-Vergne problem, 
this paper is self-contained except the results by Johnson \cite{J80a} and \S2.4.
In particular, we will give an elementary proof of the Poincar\'e-Hopf
theorem on the surface $\Sigma$ (Lemma \ref{lem:PH}).
In \S1, following Johnson \cite{J80a}, 
we study the mapping class orbits in the set of spin structures
on any compact surface $\Sigma$ with non-empty boundary. 
Generalities on framings are discussed in \S2.1.
Our computation for the case $g(\Sigma) \geq 2$ in \S2.2 
is some modification of Johnson's paper \cite{J80b}.
We need some extra invariant $\tilde A(f)$ for the case $g(\Sigma) =1$
in \S2.3. It is introduced in the end of \S2.1. In \S2.4, we prove that
the invariant $\tilde A(f)$ and the generalized Arf invariant introduced in \cite{RW14}
classify the mapping class group orbits 
in the relative genus $1$ case (Theorem \ref{relzero}).\par

In this paper we denote by $H_1(-)$ and $H^1(-)$ 
the first $\Z$-(co)homology 
%{\it integral} 
groups, and
by $H_1(-)^{(2)}$ and $H^1(-)^{(2)}$ 
the first $\Z/2$-(co)homology groups.
On $H_1(\Sigma)$ and $H_1(\Sigma)^{(2)}$, we have 
the (algebraic) intersection forms $\cdot: H_1(\Sigma)^{\otimes 2}
\to \Z$, $a\otimes b \mapsto a\cdot b$, and 
$\cdot: (H_1(\Sigma)^{(2)})^{\otimes 2}
\to \Z/2$, $a\otimes b \mapsto a\cdot b$.\par
By the classification of surfaces, any compact connected 
oriented smooth surface $\Sigma$ is classified by the genus and 
the number of the boundary components. We denote by $\Sigma_{g,n+1}$
a compact connected oriented smooth surface of genus $g$ with 
$n+1$ boundary components for $g,n \geq 0$. 
It is uniquely determined up to diffeomorphism.
Throughout this paper, we fix a system of simple closed curves 
$\{\alpha_i, \beta_i\}^g_{i=1}$ on the surface $\Sigma_{g,n+1}$
shown in Figure 1. By $\Sigma_{g,0}$ we mean a closed connected 
oriented surface of genus $g$.
\par
\bigskip
%\bigskip
This paper is a byproduct of the author's joint work with
Anton Alekseev, Yusuke Kuno and Florian Naef.
In particular, it has its origin in Alekseev's question to the author.
First of all the author thanks all of them for helpful discussions.
Furthermore Kuno kindly prepared all the figures in this paper.
After the first draft of this paper was uploaded at the arXiv, 
Oscar Randal-Williams let the author know his results in \cite{RW14}.
The author thanks him for informing about them.
The author also thanks Andrew Putman for his comments on this work.
%\thanks{
The present research is partially supported by the Grant-in-Aid for Scientific Research 
(S) (No.24224002) and (B) (No.15H03617) 
from the Japan Society for Promotion of Sciences. 
%}
%\bigskip
%\newpage
\tableofcontents
%\newpage
\section{Spin structures}

In this section, following Johnson \cite{J80a}, 
we compute the mapping class group orbits 
in the set of spin structures on any compact connected oriented 
surface $\Sigma$ with non-empty boundary $\partial\Sigma$.
\par
A spin structure on $\Sigma$ is, by definition, an unramified double
covering of the unit tangent bundle $U\Sigma$ whose restriction 
to each fiber is non-trivial. In a natural way, the set (of isomorphism 
classes) of such double coverings is isomorphic to the complement
$H^1(U\Sigma)^{(2)}\setminus H^1(\Sigma)^{(2)}$ in the exact sequence 
\begin{equation}
0 \to H^1(\Sigma)^{(2)} \overset{\varpi^*}\longrightarrow H^1(U\Sigma)^{(2)} 
\overset{\iota^*}\longrightarrow \Z/2 \to 0
\label{seq:HUSigma2}
\end{equation}
associated with the fibration $S^1\overset{\iota}\hookrightarrow U\Sigma
\overset{\varpi}\to \Sigma$. 
Here we identify $H^1(\Sigma)^{(2)}$ with its image under $\varpi^*$. 
The canonical lifting 
\begin{equation}
H_1(\Sigma)^{(2)} \to H_1(U\Sigma)^{(2)}, \quad a \mapsto \widetilde{a},
\end{equation}
is constructed in the same way as the original one for a closed surface 
by Johnson \cite{J80a}. 
In particular, if $\gamma: \coprod^m_{i=1} S^1 \to \Sigma$ is a smooth embedding,
then we have
\begin{equation}
\widetilde{[\gamma]} = \Vec{\gamma} + m\iota_*(1) \in H_1(U\Sigma)^{(2)},
\end{equation}
where $\Vec{\gamma}: \coprod^m_{i=1} S^1 \to U\Sigma$ is the (normalized)
velocity vector of $\gamma$, and $\iota_*$ is the dual of $\iota^*$ 
in the sequence (\ref{seq:HUSigma2}). As was shown in Theorem 1B in \cite{J80a},
we have $\widetilde{(a+b)} = \widetilde{a} + \widetilde{b} + (a\cdot b)\iota_*(1)$
for $a, b \in H_1(\Sigma)^{(2)}$. 
For any $\xi$ in the complement $H^1(U\Sigma)^{(2)}\setminus H^1(\Sigma)^{(2)}$, 
a quadratic form $\omega_\xi: 
H_1(\Sigma)^{(2)} \to \Z/2$ is defined by $\omega_\xi(a) 
:= \langle\xi, \widetilde{a}\rangle \in \Z/2$ for $a \in H_1(\Sigma)^{(2)}$.
By a {\it quadratic form} %on $H_1(\Sigma)^{(2)}$
we mean a function $H_1(\Sigma)^{(2)} \to \Z/2$ satisfying 
$\omega(a+b) = \omega(a) + \omega(b)+ a\cdot b$ for any 
$a$ and $b \in H_1(\Sigma)^{(2)}$. We denote by $\operatorname{Quad}(\Sigma)$ 
the set of quadratic forms on on $H_1(\Sigma)^{(2)}$. 
We remark $\omega_2 - \omega_1: H_1(\Sigma)^{(2)}
\to \Z/2$ is a homomorphism, so that it can be regarded as 
an element of $H^1(\Sigma)^{(2)}$ for any $\omega_1$ and $\omega_2
\in \operatorname{Quad}(\Sigma)$. More precisely, 
the group $H^1(\Sigma)^{(2)}$ acts on the set $\operatorname{Quad}(\Sigma)$
freely and transitively, i.e., 
the set $\operatorname{Quad}(\Sigma)$ is an affine set 
modeled by the abelian group $H^1(\Sigma)^{(2)}$. 
The mapping class group $\mathcal{M}(\Sigma)$ acts on the sets
$H^1(U\Sigma)^{(2)}\setminus H^1(\Sigma)^{(2)}$ and 
$\operatorname{Quad}(\Sigma)$
in a natural way. The map $\xi \mapsto \omega_\xi$ 
defines an $\mathcal{M}(\Sigma)$-equivariant isomorphism between 
the sets $H^1(U\Sigma)^{(2)}\setminus H^1(\Sigma)^{(2)}$ and $\operatorname{Quad}(\Sigma)$.\par
For the rest of this section we compute 
the mapping class group orbits 
in the set of quadratic forms, $\operatorname{Quad}(\Sigma)$. 
We begin by recalling some elementary facts on the (co)homology 
of the surface $\Sigma$. The cohomology exact sequence 
\begin{equation}
H^1(\Sigma, \partial\Sigma)^{(2)} \overset{j^*}\longrightarrow
H^1(\Sigma)^{(2)} \overset{i^*}\longrightarrow
H^1(\partial\Sigma)^{(2)} 
\label{eq:exact2}
\end{equation}
is compatible with the action of the mapping class group
$\mathcal{M}(\Sigma)$. In particular, the subgroup 
$\image j^* = \ker i^* \subset H^1(\Sigma)^{(2)}$ is 
stable under the action of $\mathcal{M}(\Sigma)$, and 
equals the image of the map $H_1(\Sigma)^{(2)} \to 
H^1(\Sigma)^{(2)}$, $x \mapsto x\cdot$, from the 
Poincar\'e-Lefschetz duality. 
\begin{lem}\label{lem:SCC2} 
Any homology class in $H_1(\Sigma)^{(2)}$ is represented 
by a simple closed curve.
\end{lem}
\begin{proof} The four elements in $H_1(\Sigma_{1,0})^{(2)}$ 
are represented by simple closed curves. Similarly all elements 
in $H_1(\Sigma_{0,n+1})^{(2)}$ are represented by 
simple closed curves. Any element in $H_1(\Sigma_{g,n+1})^{(2)}$
can be represented by the connected sum of some of these 
elements. This proves the lemma.
\end{proof}
For any $a \in H_1(\Sigma)^{(2)}$ we introduce a map 
$T_a: H_1(\Sigma)^{(2)} \to H_1(\Sigma)^{(2)}$ defined
by $x \mapsto x - (x\cdot a)a$. If $\gamma$ represents 
the element $a$, the map $T_a$ is induced by the right-handed
Dehn twist along $\gamma$ denoted by $t_\gamma\in \mathcal{M}(\Sigma)$. 
In particular, $T_a$ respects
the intersection form. We denote by $G(\Sigma) \subset 
\operatorname{Aut}(H_1(\Sigma)^{(2)})$ the subgroup 
generated by $\{T_a; \,\, a \in H_1(\Sigma)^{(2)}\}$.
From the Dehn-Lickorish theorem and Lemma \ref{lem:SCC2},
it equals the image of the mapping class group 
$\mathcal{M}(\Sigma)$ in the group 
$\operatorname{Aut}(H_1(\Sigma)^{(2)})$. 
In particular, the $\mathcal{M}(\Sigma)$-orbits in 
the set $\operatorname{Quad}(\Sigma)$ are the same as
the $G(\Sigma)$-orbits.
\par
For a quadratic form $\omega: H_1(\Sigma)^{(2)}
\to \Z/2$, we define a map $m_\omega: G(\Sigma) 
\to H^1(\Sigma)^{(2)}$ by $S \mapsto m_\omega(S) := 
\omega S - \omega$. Then we have
\begin{equation}\label{eq:cocyclem}
m_\omega(S_1S_2) = m_\omega(S_1)S_2 + m_\omega(S_2)
\end{equation}
for any $S_1$ and $S_2 \in G(\Sigma)$. One can compute 
$\langle m_\omega(T_a), x\rangle = \omega(T_ax) 
- \omega(x) = \omega(x-(x\cdot a)a) - \omega(x)
= (x\cdot a)\omega(a) + (x\cdot a)^2 = (x\cdot a)(\omega(a)
+ 1)$ for $a, x \in H_1(\Sigma)^{(2)}$. This means
\begin{equation}\label{eq:omegaa}
m_\omega(T_a) = (\omega(a)+1)a\cdot \in \image j^* 
\subset H^1(\Sigma)^{(2)}
\end{equation}
Hence we obtain a $1$-cocycle $m_\omega: G(\Sigma) 
\to \image j^* (\subset H^1(\Sigma)^{(2)})$. 

\begin{thm}\label{thm:orbitA}
Let $\omega_1$ and $\omega_2: H_1(\Sigma)^{(2)} \to \Z/2$
be quadratic forms. Then $\omega_1$ and $\omega_2$ belong 
to the same $\mathcal{M}(\Sigma)$-orbit if and only if
$$
\exists x \in H_1(\Sigma)^{(2)} \,\,\text{s.t.} \quad
\omega_1(x) = 0, \quad \omega_2-\omega_1 = x\cdot
\in \image j^*
\eqno{(\sharp)}
$$
\end{thm}
\begin{proof} We denote by $\omega_1\sim\omega_2$ 
the assertion that $\omega_1$ and $\omega_2$ satisfy 
the condition ($\sharp$), and begin the proof 
by checking that the relation $\sim$ is an equivalence relation 
on the set $\operatorname{Quad}(\Sigma)$. 
The reflexivity $\omega\sim \omega$ follows from $\omega(0) = 0$. 
If $x \in H_1(\Sigma)^{(2)}$ satisfies $\omega_1(x) = 0$, 
then we have $(\omega_1 + x\cdot)(x) = \omega_1(x) + x\cdot x
= 0$, which proves the symmetry: $(\omega_1\sim \omega_2)
\Rightarrow (\omega_2\sim \omega_1)$. Assume $\omega_1
\sim \omega_2$ and $\omega_2\sim\omega_3$. This means
there exist $x_1$ and $x_2 \in H_1(\Sigma)^{(2)}$ such that 
$\omega_1(x_1) = \omega_2(x_2) = 0$, $\omega_2-\omega_1
= x_1\cdot$ and $\omega_3-\omega_2= x_2\cdot$. 
Then we have $\omega_3-\omega_1 = (x_1+x_2)\cdot$ 
and $\omega_1(x_1+x_2) = \omega_1(x_1) + x_1\cdot x_2
+ \omega_1(x_2) = \omega_1(x_1) + \omega_2(x_2) = 0$.
Hence we obtain $\omega_1\sim\omega_3$.
This proves the transitivity.\par
Next we assume $\omega_2 = \omega_1T_a$ for some $a \in 
H_1(\Sigma)^{(2)}$. Then, by the formula (\ref{eq:omegaa}), 
we have $\omega_2-\omega_1 = m_{\omega_1}(T_a) 
= (\omega_1(a)+1)a\cdot$, while $\omega_1((\omega_1(a)+1)a)
= (\omega_1(a)+1)\omega_1(a) = 0$. This implies $\omega_1
\sim \omega_1T_a$. The relation $\sim$ is an equivalence 
relation, and $G(\Sigma)$ is generated by $T_a$'s.
Hence, if $\omega_1$ and $\omega_2$ belong 
to the same $G(\Sigma)$-orbit, then we have $\omega_1
\sim \omega_2$. \par
Conversely, if there exists some $x \in H_1(\Sigma)^{(2)}$
such that $\omega_1(x) = 0$ and $\omega_2-\omega_1
= x\cdot$. Then we have $\omega_1T_x - \omega_1
= m_{\omega_1}(T_x) = (\omega_1(x)+1)x\cdot
= x\cdot = \omega_2-\omega_1$, so that $\omega_2
= \omega_1T_x$. In particular, 
$\omega_1$ and $\omega_2$ belong 
to the same $G(\Sigma)$-orbit.\par
This completes the proof of the theorem.
\end{proof}

Now consider the inclusion homomorphism $i_*: H_1(\partial\Sigma)^{(2)}
\to H_1(\Sigma)^{(2)}$. Any $\omega \in \operatorname{Quad}(\Sigma)$ 
restricts to a homomorphism on $H_1(\Sigma)^{(2)}$ via the homomorphism
$i_*$, since the intersection form vanishes on $i_*H_1(\partial\Sigma)^{(2)}$. 
Hence we have the restriction map
\begin{equation}\label{eq:restquad}
i^*: \operatorname{Quad}(\Sigma) \to H^1(\partial\Sigma)^{(2)}, \quad
\omega \mapsto i^*\omega = \omega\circ i_*.
\end{equation}
The kernel $\ker i_*$ is spanned by the $\Z/2$-fundamental class
$[\partial\Sigma]_2 \in H_1(\partial\Sigma)^{(2)}$. Hence, if 
$h \in H^1(\partial\Sigma)^{(2)}$ satisfies $h[\partial\Sigma]_2 = 0$, 
then it induces a homomorphism on $i_*H_1(\partial\Sigma)^{(2)}$,
and extended to the element of $H^1(\Sigma)^{(2)}$ satisfying 
$h([\alpha_i]) = h([\beta_i]) = 0$ for any $1 \leq i \leq g$. 
Here $\alpha_i$ and $\beta_i$ are the simple closed curves 
shown in Figure 1. Moreover we define a map $\omega^{0,h}: 
H_1(\Sigma)^{(2)} \to \Z/2$ by 
\begin{equation}\label{eq:omegaz}
\omega^{0,h}(x) := \sum^g_{i=1}(x\cdot [\alpha_i])(x\cdot[\beta_i]) 
+h(x)
\end{equation}
for $x \in H_1(\Sigma)^{(2)}$. 
It is easy to check $\omega^{0,h}$ is a quadratic form, and 
$i^*\omega^{0,h} = h$. 
If a quadratic form $\omega \in \operatorname{Quad}(\Sigma)$ 
satisfies $i^*\omega = 0 \in H^1(\partial\Sigma)^{(2)}$, 
then the Arf invariant $\operatorname{Arf}(\omega)$ is defined by 
\begin{equation}\label{eq:defArf}
\operatorname{Arf}(\omega) := \sum^g_{i=1}
\omega([\alpha_i])\omega([\beta_i]) \in \Z/2
\end{equation}
\cite{A41}.
For any $x \in H_1(\Sigma)^{(2)}$, we have 
\begin{equation}\label{eq:xArf}
\operatorname{Arf}(\omega^{0,0} + x\cdot) = \sum^g_{i=1}
(x\cdot[\alpha_i])(x\cdot[\beta_i]) = \omega^{0,0}(x).
\end{equation}
In particular, the Arf invariant $\operatorname{Arf}$ is $G(\Sigma)$-invariant, 
namely, we have $\operatorname{Arf}(\omega S) = \operatorname{Arf}(\omega)$
for any $\omega \in (i^*)^{-1}(0)$ and $S \in G(\Sigma)$. 
In fact, there are $x_0$ and $x_1 \in H_1(\Sigma)^{(2)}$ such that 
$\omega = \omega^{0,0} + x_0\cdot$, $\omega S- \omega = x_1\cdot$ 
and $\omega(x_1) = 0$. Then we have $\operatorname{Arf}(\omega S) 
= \omega^{0,0}(x_0+x_1) 
= \omega^{0,0}(x_0) + x_0\cdot x_1 + \omega^{0,0}(x_1)
= \operatorname{Arf}(\omega) + \omega(x_1) 
= \operatorname{Arf}(\omega)$.\par
Now recall $m_\omega(G(\Sigma)) \subset \ker(i^*: 
H^1(\Sigma)^{(2)} \to H^1(\partial\Sigma)^{(2)})$ and 
$G(\Sigma)$ is the image of $\mathcal{M}(\Sigma)$ 
in $\aut(H_1(\Sigma)^{(2)})$. Hence the restriction map
$i^*$ induces the map
\begin{equation}\label{eq:rho2}
\rho_2: \operatorname{Quad}(\Sigma)/\mathcal{M}(\Sigma)
\to H^1(\partial\Sigma)^{(2)}, \quad
\omega \bmod G(\Sigma) \mapsto i^*\omega.
\end{equation}
\begin{thm}\label{thm:mcgspin}
For any $h \in H^1(\partial\Sigma)^{(2)}$, the cardinality of 
the set ${\rho_2}^{-1}(h)$ is given by
$$
\sharp {\rho_2}^{-1}(h) = 
\begin{cases}
0, & \text{if $h[\partial\Sigma]_2 \neq 0$,}\\
1, & \text{if $h[\partial\Sigma]_2 = 0$ and ($h \neq 0$ or $g = 0$),}\\
2, & \text{if $h= 0$ and $g \geq 1$.}\\
\end{cases}
$$
In the last case, the two orbits are distinguished by the Arf invariant
$\operatorname{Arf}: (i^*)^{-1}(0) \to \Z/2$.
\end{thm}
\begin{proof} (0) If $h[\partial\Sigma]_2 \neq 0$, we have 
$(i^*)^{-1}(h) = \emptyset$ since $i_*[\partial\Sigma]_2 = 0$.\par
(1) Suppose $h[\partial\Sigma]_2 = 0$ and $g = 0$. 
Then $(i^*)^{-1}(h) = \{\omega^{0,h}\}$ is a one-point set.\par
Next suppose $h[\partial\Sigma]_2 = 0$, $h \neq 0$ and $g \geq 1$. 
Then $\omega^{0,h} \in (i^*)^{-1}(h) \neq \emptyset$. 
For any $\omega \in (i^*)^{-1}(h)$ we have $\omega - \omega^{0,h}
\in \ker i^* = \image j^*$, so that $\omega - \omega^{0,h}
= x_0\cdot \in H^1(\Sigma)^{(2)}$ for some $x_0 \in H_1(\Sigma)^{(2)}$.
Since $h \neq 0$, we have $\omega(x_0) = h(x_1)$ for some $x_1 
\in H_1(\partial\Sigma)^{(2)}$. 
Then $(x_0+x_1)\cdot = x_0\cdot = \omega - \omega^{0,h}$ and 
$\omega(x_0+x_1) = \omega(x_0) + x_0\cdot x_1 + \omega(x_1)
= h(x_1) + 0 + h(x_1) = 0$. By Theorem \ref{thm:orbitA}, 
we have $\omega = \omega^{0,h}S$ for some $S \in G(\Sigma)$.
This proves $\sharp {\rho_2}^{-1}(h) = 1$.\par
(2) Suppose $h = 0$ and $g \geq 1$. 
Then $\omega^{0,0} \in (i^*)^{-1}(0) \neq \emptyset$, and 
we have $\omega^{0,0}(x_0) = 1$ for some $x_0 \in H_1(\Sigma)^{(2)}$.
For any $\omega \in (i^*)^{-1}(0)$ 
there exists some $x \in H_1(\Sigma)^{(2)}$ 
such that $\omega - \omega^{0,0} = x\cdot \in H^1(\Sigma)^{(2)}$.
If $\omega^{0,0}(x) = \operatorname{Arf}(\omega) = 0$, 
then, by Theorem \ref{thm:orbitA},
we have $\omega = \omega^{0,0}S$ for some $S \in G(\Sigma)$.
On the other hand, 
if $\omega^{0,0}(x) = \operatorname{Arf}(\omega) = 1$, 
then we have $\omega - (\omega^{0,0} + x_0\cdot) = 
(x-x_0)\cdot$ and $(\omega^{0,0} + x_0\cdot)(x-x_0) 
= \omega^{0,0}(x-x_0) + x_0\cdot x = 
\omega^{0,0}(x) - \omega^{0,0}(x_0) = 0$.
This implies $\omega = (\omega^{0,0} + x_0\cdot)S$ 
for some $S \in G(\Sigma)$.
This proves $\sharp {\rho_2}^{-1}(0) = 2$.\par
This completes the proof of the theorem.
\end{proof}

As was proved by Randal-Williams in \cite{RW14} Theorem 2.9, 
the cardinality of the mapping class group orbit sets 
in the set of spin structures
for the relative version 
is always $2$, and does not depend on the boundary value.
In particular, the (generalized) Arf invariant can be defined in any cases.
The situation is similar for framings in the case $g \geq 2$ 
(Theorem \ref{thm:genus2}).

\section{Framings}
\subsection{Generalities}
Let $\Sigma$ be a compact connected oriented smooth surface 
with non-empty boundary as before. 
In this paper, we denote by $F(\Sigma)$ the set of homotopy classes 
of {\it framings}, i.e., orientation-preserving global trivializations
$T\Sigma \overset\cong\to \Sigma\times\mathbb{R}^2$ of the tangent 
bundle $T\Sigma$. In this paper, the composite of such an trivialization and the 
second projection,  $T\Sigma\overset\cong\to \Sigma\times\mathbb{R}^2
\overset{\operatorname{pr}_2}\to \mathbb{R}^2$, is also called a framing.
The group $[\Sigma, S^1] = H^1(\Sigma) = H^1(\Sigma; \Z)$ acts on 
the set $F(\Sigma)$ freely and transitively. In fact, the difference of any two framings
gives a continuous map $\Sigma \to GL^+(2; \mathbb{R}) \simeq S^1$. 
The mapping class group $\mathcal{M}(\Sigma)$ acts on the set $F(\Sigma)$ 
from the right in a natural way. 
\par
Consider the inclusion map $\iota: S^1 \hookrightarrow U\Sigma$
and the projection  $\varpi: U\Sigma \to \Sigma$ as in the preceding section.
Then we have $\mathcal{M}(\Sigma)$-equivariant exact sequences
\begin{eqnarray}
&& 0 \to \Z \overset{\iota_*}\longrightarrow H_1(U\Sigma) \overset{\varpi_*}\longrightarrow H^1(\Sigma)\to 0, \quad\text{and}
\label{seq:HUSigma}\\
&& 0 \to H^1(\Sigma) \overset{\varpi^*}\longrightarrow H^1(U\Sigma)
\overset{\iota^*}\longrightarrow \Z \to 0
\label{seq:coHUSigma}
\end{eqnarray}
in the integral (co)homology. 
The group $H^1(\Sigma)$ obviously acts on the inverse image $(\iota^*)^{-1}(1)$ of 
$1 \in \Z$ freely and transitively. For a framing $f \in F(\Sigma)$ 
we denote by $\xi(f) \in H^1(U\Sigma)$ 
the pull-back of the positive generator of $H^1(S^1)$ 
by the map $f: U\Sigma \to S^1$. It is clear that $\iota^*\xi(f) = 1 \in \Z$. 
Then the map $F(\Sigma) \to (\iota^*)^{-1}(1)$, $f\mapsto \xi(f)$, is 
equivariant under the actions of the groups $\mathcal{M}(\Sigma)$ and
$H^1(\Sigma)$. In particular, it is an $\mathcal{M}(\Sigma)$-equivariant
isomorphism $F(\Sigma) \cong (\iota^*)^{-1}(1)$, by which we identify
these two sets with each other.\par
An immersion $\ell: S^1 \to \Sigma$ lifts to its (normalized)
velocity vector $\Vec{\ell}: S^1 \to U\Sigma$, $t \mapsto
\dot{\ell}(t)/\Vert\dot{\ell}(t)\Vert$. 
The rotation number of $\ell$ with respect to a framing $f$ is defined by
\begin{equation}\label{eq:defrot}
\rot_f\ell := \langle\xi(f), [\Vec{\ell}]\rangle
= \deg(f\circ \Vec{\ell}: S^1 \to S^1) \in \Z
\end{equation}
\cite{W37}. 
For any $\varphi \in \mathcal{M}(\Sigma)$ we have
\begin{equation}\label{eq:actrot}
\rot_{f\circ\varphi}(\ell) = \rot_f(\varphi\circ\ell).
\end{equation}
\begin{lem}\label{lem:betti}
 If $\ell_i: S^1 \to \Sigma$, $1 \leq i \leq b_1 = b_1(\Sigma)$, 
is an immersion, and the set $\{[\ell_i]\}^{b_1}_{i=1}$ constitutes 
a free basis of $H_1(\Sigma)$, then the map 
$$
F(\Sigma) \to \Z^{b_1}, \quad
f \mapsto (\rot_f(\ell_i))^{b_1}_{i=1}
$$
is a bijection.
\end{lem}
\begin{proof} Then the set $\{[\Vec{\ell_i}]\}^{b_1}_{i=1}\cup
\{\iota_*(1)\}$ constitutes a free basis of $H_1(U\Sigma)$.
\end{proof}
The mod $2$ reduction of $\xi(f)$, which we denote by 
$\xi_2(f) \in H^1(U\Sigma)^{(2)}$, is a spin structure on the surface $\Sigma$.
We write simply $\omega_f := \omega_{\xi_2(f)}: H_1(\Sigma)^{(2)} \to 
\Z/2$ for the corresponding quadratic form. 
\begin{lem}\label{lem:omegarot} For any smooth embedding 
$\ell: S^1 \to \Sigma$, we have 
$$
\omega_f([\ell]) = \rot_f(\ell) + 1\in \Z/2.
$$
\end{lem}
\begin{proof} Recall the canonical lifting in \cite{J80a} is given by 
$\widetilde{[\ell]} = [\Vec{\ell}] + \iota_*(1) \in H_1(U\Sigma)^{(2)}$. 
Hence we have 
$$
\omega_f([\ell]) = \langle\xi_2(f), \widetilde{[\ell]}\rangle 
= \langle\xi_2(f), [\Vec{\ell}]\rangle + 1
= \rot_f(\ell) + 1 \in \Z/2.
$$
This proves the lemma.
\end{proof}
The following is a straight-forward consequence of the Poincar\'e-Hopf theorem.
But we will give its elementary proof for the convenience of non-experts on topology.
\begin{lem}\label{lem:PH}
Let $S \subset \Sigma$ be a compact smooth subsurface. 
We number the boundary components of $S$: 
$\partial S = \coprod^N_{k=1}\partial_kS$. Then we have 
$$
\sum^N_{k=1}\rot_f(\partial_kS) = \chi(S)
$$
for any $f \in F(\Sigma)$. Here we endow each $\partial_kS$ 
with the orientation induced by $S$, and $\chi(S)$ is the 
Euler characteristic of the surface $S$.
\end{lem}
\begin{proof}
Let $\{(e_\lambda, \varphi_\lambda: D^{n_\lambda} \to S)\}_{\lambda\in 
\Lambda}$ be a finite cell decomposition of the surface $S$ such that
each characteristic map $\varphi_\lambda: D^{n_\lambda} \to 
\overline{e_\lambda} \subset S$
is a smooth embedding and each $0$-cell is located on the boundary $\partial S$.
We denote $C_i := \sharp\{\lambda\in \Lambda; \,\, n_\lambda = i\}$,
$0\leq i \leq 2$, so that $\chi(S) = C_2 - C_1 + C_0$.
Then we compute the sum $\sum_{n_\lambda=2} \rot_f(\varphi_\lambda(\partial 
D^{n_\lambda}))$. Since the loop $\varphi_\lambda(\partial D^{2})$ is 
regular homotopic to a small loop around the center of $e_\lambda$, 
the sum equals $C_2$. The contribution of both sides of each interior $1$-cell 
cancel each other, while the contribution of the boundary $1$-cells equals
the sum $\sum^N_{k=1}\rot_f(\partial_kS)$. The contribution of a vertex, i.e., 
a $0$-cell $e_\lambda$ equals $\dfrac12(d_\lambda -2)$, where $d_\lambda$ 
is the valency at the vertex $e_\lambda$. See Figure 2.
On the other hand, we have $C_1 = \dfrac12\sum_{n_\lambda=0}
d_\lambda$. Hence we obtain 
$$
C_2 = \left(\sum^N_{k=1}\rot_f(\partial_kS)\right) 
+ \dfrac12\sum_{n_\lambda=0}(d_\lambda-2)
= \left(\sum^N_{k=1}\rot_f(\partial_kS)\right) + C_1 - C_0,
$$
which proves the lemma.
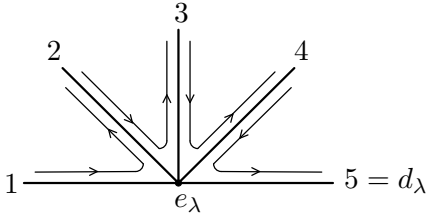
\begin{figure}
\begin{center}
%WinTpicVersion4.32a
{\unitlength 0.1in%
\begin{picture}(17.6000,9.9200)(1.9800,-11.0000)%
% DOT 0 0 3 0 Black White  
% 1 1100 1100
% 
\special{pn 4}%
\special{sh 1}%
\special{ar 1100 1100 16 16 0 6.2831853}%
% LINE 1 0 3 0 Black White  
% 2 1100 1100 1100 300
% 
\special{pn 13}%
\special{pa 1100 1100}%
\special{pa 1100 300}%
\special{fp}%
% LINE 1 0 3 0 Black White  
% 2 300 1100 1900 1100
% 
\special{pn 13}%
\special{pa 300 1100}%
\special{pa 1900 1100}%
\special{fp}%
% LINE 1 0 3 0 Black White  
% 2 1100 1100 500 500
% 
\special{pn 13}%
\special{pa 1100 1100}%
\special{pa 500 500}%
\special{fp}%
% LINE 2 0 3 0 Black White  
% 2 1000 920 600 520
% 
\special{pn 8}%
\special{pa 1000 920}%
\special{pa 600 520}%
\special{fp}%
% LINE 2 0 3 0 Black White  
% 2 1040 880 1040 360
% 
\special{pn 8}%
\special{pa 1040 880}%
\special{pa 1040 360}%
\special{fp}%
% CIRCLE 2 0 3 0 Black White  
% 4 1000 880 1000 920 1000 920 1040 880
% 
\special{pn 8}%
\special{ar 1000 880 40 40 6.2831853 1.5707963}%
% LINE 2 0 3 0 Black White  
% 2 918 1000 516 602
% 
\special{pn 8}%
\special{pa 918 1000}%
\special{pa 516 602}%
\special{fp}%
% LINE 2 0 3 0 Black White  
% 2 879 1040 358 1043
% 
\special{pn 8}%
\special{pa 879 1040}%
\special{pa 358 1043}%
\special{fp}%
% CIRCLE 2 0 3 0 Black White  
% 4 878 1000 918 1000 879 1040 918 1000
% 
\special{pn 8}%
\special{ar 878 1000 40 40 6.2831853 1.5458015}%
% LINE 1 0 3 0 Black White  
% 2 1100 1100 1700 500
% 
\special{pn 13}%
\special{pa 1100 1100}%
\special{pa 1700 500}%
\special{fp}%
% LINE 2 0 3 0 Black White  
% 2 1282 1000 1684 602
% 
\special{pn 8}%
\special{pa 1282 1000}%
\special{pa 1684 602}%
\special{fp}%
% LINE 2 0 3 0 Black White  
% 2 1321 1040 1842 1043
% 
\special{pn 8}%
\special{pa 1321 1040}%
\special{pa 1842 1043}%
\special{fp}%
% CIRCLE 2 0 3 0 Black White  
% 4 1322 1000 1282 1000 1282 1000 1321 1040
% 
\special{pn 8}%
\special{ar 1322 1000 40 40 1.5957911 3.1415927}%
% LINE 2 0 3 0 Black White  
% 2 1200 920 1600 520
% 
\special{pn 8}%
\special{pa 1200 920}%
\special{pa 1600 520}%
\special{fp}%
% LINE 2 0 3 0 Black White  
% 2 1160 880 1160 360
% 
\special{pn 8}%
\special{pa 1160 880}%
\special{pa 1160 360}%
\special{fp}%
% CIRCLE 2 0 3 0 Black White  
% 4 1200 880 1200 920 1160 880 1200 920
% 
\special{pn 8}%
\special{ar 1200 880 40 40 1.5707963 3.1415927}%
% LINE 2 0 3 0 Black White  
% 4 680 1040 640 1060 680 1040 640 1020
% 
\special{pn 8}%
\special{pa 680 1040}%
\special{pa 640 1060}%
\special{fp}%
\special{pa 680 1040}%
\special{pa 640 1020}%
\special{fp}%
% LINE 2 0 3 0 Black White  
% 4 1520 1040 1480 1060 1520 1040 1480 1020
% 
\special{pn 8}%
\special{pa 1520 1040}%
\special{pa 1480 1060}%
\special{fp}%
\special{pa 1520 1040}%
\special{pa 1480 1020}%
\special{fp}%
% LINE 2 0 3 0 Black White  
% 4 1160 680 1140 640 1160 680 1180 640
% 
\special{pn 8}%
\special{pa 1160 680}%
\special{pa 1140 640}%
\special{fp}%
\special{pa 1160 680}%
\special{pa 1180 640}%
\special{fp}%
% LINE 2 0 3 0 Black White  
% 4 1040 640 1020 680 1040 640 1060 680
% 
\special{pn 8}%
\special{pa 1040 640}%
\special{pa 1020 680}%
\special{fp}%
\special{pa 1040 640}%
\special{pa 1060 680}%
\special{fp}%
% LINE 2 0 3 0 Black White  
% 4 738 820 780 837 738 820 749 863
% 
\special{pn 8}%
\special{pa 738 820}%
\special{pa 780 837}%
\special{fp}%
\special{pa 738 820}%
\special{pa 749 863}%
\special{fp}%
% LINE 2 0 3 0 Black White  
% 4 862 786 820 769 862 786 851 743
% 
\special{pn 8}%
\special{pa 862 786}%
\special{pa 820 769}%
\special{fp}%
\special{pa 862 786}%
\special{pa 851 743}%
\special{fp}%
% LINE 2 0 3 0 Black White  
% 4 1382 740 1340 757 1382 740 1371 783
% 
\special{pn 8}%
\special{pa 1382 740}%
\special{pa 1340 757}%
\special{fp}%
\special{pa 1382 740}%
\special{pa 1371 783}%
\special{fp}%
% LINE 2 0 3 0 Black White  
% 4 1418 866 1460 849 1418 866 1429 823
% 
\special{pn 8}%
\special{pa 1418 866}%
\special{pa 1460 849}%
\special{fp}%
\special{pa 1418 866}%
\special{pa 1429 823}%
\special{fp}%
% STR 2 0 3 0 Black White  
% 4 198 1130 198 1150 2 0 0 0
% $1$
\put(1.9800,-11.5000){\makebox(0,0)[lb]{$1$}}%
% STR 2 0 3 0 Black White  
% 4 418 440 418 460 2 0 0 0
% $2$
\put(4.1800,-4.6000){\makebox(0,0)[lb]{$2$}}%
% STR 2 0 3 0 Black White  
% 4 1078 240 1078 260 2 0 0 0
% $3$
\put(10.7800,-2.6000){\makebox(0,0)[lb]{$3$}}%
% STR 2 0 3 0 Black White  
% 4 1698 440 1698 460 2 0 0 0
% $4$
\put(16.9800,-4.6000){\makebox(0,0)[lb]{$4$}}%
% STR 2 0 3 0 Black White  
% 4 1958 1130 1958 1150 2 0 0 0
% $5=d_{\lambda}$
\put(19.5800,-11.5000){\makebox(0,0)[lb]{$5=d_{\lambda}$}}%
% STR 2 0 3 0 Black White  
% 4 1078 1230 1078 1250 2 0 0 0
% $e_{\lambda}$
\put(10.7800,-12.5000){\makebox(0,0)[lb]{$e_{\lambda}$}}%
\end{picture}}%
\end{center}
\caption{the case $d_{\lambda}=5$}
\end{figure}
%We begin by studying the annulus $A := \{z \in \mathbb{C}; \,\,
%\vert z\vert \geq 1\}$ and a nowhere vanishing vector field $X$ on $A$.
%We can deform $X$ near the boundary $\partial A = \{\vert z\vert = 1\}$
%such that $X$ equals $z^m\dfrac{d}{dz}$ for some $m \in \Z$ 
%near $\partial A$. The loop $\ell_1: z \in S^1 = \partial A
%\mapsto z \in \partial A$ goes around the boundary negatively.
%If $f_X$ is the framing induced by $X$, then we have 
%$\rot_{f_X}(\partial A) = - \rot_{f_X}(\ell_1) = -\deg(S^1 \to S^1,
%z \mapsto \sqrt{-1}z/z^m) = m-1$. \par
%Conversely if a framing $f_0$ on $A$ satisfies $\rot_{f_0}(\partial A)
%= \nu - 1 \in \Z$, then the vector field ${f_0}^{-1}(1,0)$ extends to 
%the whole $\mathbb{C}$ such that it has a unique zero of order 
%$\nu$ at the origin.\par
%Now we may assume $S$ is diffeomorphic to $\Sigma_{h,N}$ 
%for some $h \geq 0$. We write $\nu_k := \rot_f(\partial_kS) + 1$. 
%We glue a $2$-disk $D_k \cong D^2$ along each $\partial_kS$ 
%to get the closed surface $\Sigma_{h,0}$. The vector field 
%$f^{-1}(1,0)$ extends the the whole $\Sigma_{h,0}$ such that 
%it has a zero of order $\nu_k$ at the origin of each $D_k$, 
%$1 \leq k \leq N$. By the Poincar\'e-Hopf theorem 
%for the closed surface $\Sigma_{h,0}$, we have 
%$\sum^N_{k=1}\nu_k = \chi(\Sigma_{h,0}) = 2-2h$, 
%so that $\sum^N_{k=1}\rot_f(\partial_kS) = 2-2h - N = \chi(S)$.
%This proves the lemma.
\end{proof}
\par
Now suppose $\Sigma = \Sigma_{g, n+1}$ for $g,n \geq 0$. 
We number the boundary components: 
$\partial \Sigma = \coprod^n_{i=0}\partial_i\Sigma$.
Since any element of the group $\mathcal{M}(\Sigma)$ 
fixes the boundary pointwise, we can define a map
\begin{equation}\label{eq:defrho}
\rho: F(\Sigma)/\mathcal{M}(\Sigma) \to \Z^{n+1}, \quad
f \bmod \mathcal{M}(\Sigma) \mapsto 
(\rot_f(\partial_i\Sigma)+1)^n_{i=0}.
\end{equation}
Here, taking Lemma \ref{lem:omegarot} into account, 
we consider $\rot_f(\partial_i\Sigma)+1$ instead of 
the rotation number itself. 
By Lemmas \ref{lem:betti} and \ref{lem:PH}, we have 
\begin{equation}\label{eq:rimage}
\image\rho = \{(\nu_i)^n_{i=0} \in \Z^{n+1}; \,\,
{\sum}^n_{i=0}\nu_i = 2 - 2g \}.
\end{equation}
In the genus $0$ case, i.e., $\Sigma = \Sigma_{0,n+1}$, 
these lemmas imply
\begin{equation}\label{eq:genus0}
F(\Sigma)/\mathcal{M}(\Sigma) = F(\Sigma) \overset\rho\cong
\{(\nu_i)^n_{i=0} \in \Z^{n+1}; \,\,
{\sum}^n_{i=0}\nu_i = 2\}.
\end{equation}

We conclude this subsection by introducing an extra invariant for a framing,
which will be used for the genus $1$ case.
For $f \in F(\Sigma)$ we consider the ideal 
$\mathfrak{a}(f)$ in $\Z$ generated by the set 
$\{\rot_f(\gamma); \,\, \text{$\gamma$ is a non-separating simple closed 
curve in $\Sigma$}\}$, and define
$\tilde A(f) \in \Z_{\geq0}$ to be the non-negative generator of the ideal 
$\mathfrak{a}(f)$. It is clear that these are invariants under the action 
of the mapping class group $\mathcal{M}(\Sigma)$.
But, if $g\geq 2$, they are trivial invariants.
\begin{lem}\label{lem:genA}
If $g \geq 2$, 
we have $\tilde A(f) = 1$ for any $f \in F(\Sigma_{g, n+1})$. 
\end{lem}
\begin{proof} 
From the assumption, there is a smooth compact subsurface $P 
\subset \Sigma$  diffeomorphic to a pair of pants $\Sigma_{0,3}$ 
such that each of the three boundary components $\partial_iP$, 
$0 \leq i \leq 2$, is a non-separating curve in $\Sigma$.
Then, from Lemma \ref{lem:PH}, we have 
$\rot_f(\partial_0P) + \rot_f(\partial_1P) + \rot_f(\partial_2P) 
= \chi(P) = -1$, so that $-1 \in \mathfrak{a}(f)$. 
This proves the theorem.
\end{proof}

\subsection{The case $g \geq 2$}
In this subsection we consider $\Sigma = \Sigma_{g,n+1}$ 
for the case $g \geq 2$. 
In this case our computation modifies that in \cite{J80b}. 
Consider the map $\rho: 
F(\Sigma)/\mathcal{M}(\Sigma) \to \Z^{n+1}$ 
in (\ref{eq:defrho}).
\begin{thm}\label{thm:genus2} Suppose $g \geq 2$. 
Then, for any $\nu \in \image\rho = 
\{(\nu_i)^n_{i=0} \in \Z^{n+1}; \,\,
{\sum}^n_{i=0}\nu_i = 2-2g\}$, we have 
$$
\sharp\rho^{-1}(\nu) = 
\begin{cases}
1, & \text{if $\nu \in \image\rho \setminus (2\Z)^{n+1}$,}\\
2, & \text{if $\nu \in \image\rho\cap (2\Z)^{n+1}$.}
\end{cases}
$$
In the latter case, the two orbits are distinguished by the Arf invariant
of the spin structure $\xi_2(f)$.
\end{thm}
\begin{proof} Let $f_1$ and $f_2 \in F(\Sigma)$ satisfy 
$\rho(f_1) = \rho(f_2)$ and $\operatorname{Arf}(\xi_2(f_1))
= \operatorname{Arf}(\xi_2(f_2))$ if $\rho(f_1) = \rho(f_2) 
\in (2\Z)^{n+1}$. Then, by Theorem \ref{thm:mcgspin}, 
we have 
\begin{equation}\label{eq:mod2fr}
\xi(f_2) - \xi(f_1\varphi_0) 
= 2(\sum^g_{i=1}\lambda_i[\alpha_i] + \sum^g_{i=1}\mu_i[\beta_i])\cdot 
\in H^1(\Sigma)
\end{equation}
for some $\varphi_0 \in \mathcal{M}(\Sigma)$ and 
$\lambda_i, \mu_i \in \Z$. 
Here $\alpha_i$ and $\beta_i$ are the simple closed curves 
shown in Figure 1.
Hence it suffices to construct
$\varphi'_i$ and $\varphi''_i \in \mathcal{M}(\Sigma)$ for each 
$1 \leq i \leq g$ such that 
\begin{equation}\label{eq:varphi}
\xi(f\varphi'_i) - \xi(f) = 2[\alpha_i]\cdot \quad\text{and}\quad
\xi(f\varphi''_i) - \xi(f) = 2[\beta_i]\cdot
\end{equation}
for any $f \in F(\Sigma)$.
We denote by $t_\gamma \in \mathcal{M}(\Sigma)$
the right-handed Dehn twist along a simple closed curve 
$\gamma$ in $\Sigma$.
\par
Now from the assumption $g\geq2$ there exist simple closed 
curves $\hat\alpha_i$ and $\hat\beta_i$ satisfying the conditions
\begin{enumerate}
\item[(i')] $\alpha_i$ and $\hat\alpha_i$ bound a smooth compact 
subsurface diffeomorphic to $\Sigma_{1,2}$.
\item[(i'')] $\beta_i$ and $\hat\beta_i$ bound a smooth compact 
subsurface diffeomorphic to $\Sigma_{1,2}$.
\item[(ii)] $\hat\alpha_i$ and $\hat\beta_i$ are disjoint from 
$\{\alpha_k, \beta_k\}_{k\neq i}$.
\item[(iii')] $\hat\alpha_i$ intersects with $\beta_i$ transversely
at a unique point.
\item[(iii'')] $\hat\beta_i$ intersects with $\alpha_i$ transversely
at a unique point.
\end{enumerate}
Choose a point on each component of $\partial\Sigma_{1,2}$. 
Then, by the disk theorem, two simple arcs connecting 
these two chosen points are mapped to each other by the 
action of the group $\mathcal{M}(\Sigma_{1,2})$. 
Similar transitivity holds also for the surface $\Sigma_{g-2, n+3}$. 
Hence, by the classification theorem of surfaces, the quadruples
$(\Sigma, \alpha_i, \hat\alpha_i, \beta_i)$ and 
$(\Sigma, \beta_i, \hat\beta_i, \alpha_i)$ are diffeomorphic to 
$(\Sigma, \gamma_1, \gamma_2, \gamma_0)$ in Figure 3 (a).
Then the simple closed curve ${t_{\gamma_2}}^{-1}t_{\gamma_1}
(\gamma_0)$ is computed as in Figure 3 (b), so that 
$\gamma_0$ and ${t_{\gamma_2}}^{-1}t_{\gamma_1}
(\gamma_0)$ bound a smooth compact subsurface diffeomorphic 
to $\Sigma_{1,2}$ By Lemma \ref{lem:PH}, we have 
$$
\left\vert\rot_{f{t_{\gamma_2}}^{-1}t_{\gamma_1}}(\gamma_0)
-\rot_f(\gamma_0)\right\vert = \vert\chi(\Sigma_{1,2})\vert = 2
$$
for any $f \in F(\Sigma)$. It is clear that
$\rot_{f{t_{\gamma_2}}^{-1}t_{\gamma_1}}(\gamma_1)
=\rot_f(\gamma_1)$. 
The mapping class ${t_{\gamma_2}}^{-1}t_{\gamma_1}$ is just a BP-map
in \cite{J80b}. 
\begin{figure}
\begin{center}
\input{fig2.tex}
\label{fig2}
\caption{}
\end{center}
\end{figure}
\par
Hence, if we take $\varphi'_i$ to be $t_{\hat\alpha_i}^{-1}t_{\alpha_i}$
or its inverse, then $\rot_{f\varphi'_i}(\beta_i)-\rot_f(\beta_i) = 2$ and 
$\rot_{f\varphi'_i}(\alpha_i)-\rot_f(\alpha_i) = 0$. 
From the condition (ii) above, 
$\rot_{f\varphi'_i}(\alpha_k)-\rot_f(\alpha_k) = 0$ and 
$\rot_{f\varphi'_i}(\beta_k)-\rot_f(\beta_k) = 0$ for $k\neq i$.
Hence $\xi(f\varphi'_i) - \xi(f) = 2[\alpha_i]\cdot$ as desired
in (\ref{eq:varphi}).
Similarly, if we take $\varphi''_i$ to be  
$t_{\beta_i}^{-1}t_{\beta_i}$ or its inverse, then $\varphi''_i$ 
satisfies (\ref{eq:varphi}). This proves the theorem.
\end{proof}

\subsection{The genus $1$ case}

Finally we study the genus $1$ case: $\Sigma = \Sigma_{1,n+1}$.
We write simply $\alpha = \alpha_1$ and $\beta = \beta_1$ 
shown in Figure 1, $\nu_j = \nu_j(f) := \rot_f(\partial_j\Sigma) + 1$,
$0 \leq j \leq n$, and take a closed regular neighbourhood $\Sigma'$
of the subset $\alpha(S^1)\cup \beta(S^1)$. 
It is diffeomorphic to $\Sigma_{1,1}$.
We begin by computing the invariant $\tilde A(f)$ for $f \in F(\Sigma)$.
\begin{lem}\label{lem:bf}
The ideal in $\Z$ generated by the set $\{\rot_f(\alpha), \rot_f(\beta), 
\nu_j(f); \,\, 0 \leq j \leq n\}$ equals the ideal $\mathfrak{a}(f)$.
In other words, $\tilde A(f)$ is the non-negative greatest common divisor 
of the set. 
\end{lem}
\begin{proof}
We denote the ideal given above by $\mathfrak{b}(f)$.
For each $0 \leq j \leq n$, we choose a band connecting 
$\alpha$ and $\partial_j\Sigma$ 
%and disjoint from $\beta$ 
to obtain a non-separating 
simple closed curve $\alpha^{(j)}$ such that $\alpha$, $\partial_j\Sigma$ 
and $\alpha^{(j)}$ bound a pair of pants. Then we have 
$\rot_f(\alpha^{(j)}) = \rot_f(\alpha) + \nu_j$, so that we obtain 
$\mathfrak{b}(f) \subset \mathfrak{a}(f)$.\par
Let $\gamma$ be any non-separating simple closed curve in $\Sigma$. 
When the curve $\gamma$ crosses the boundary component $\partial_j\Sigma$, 
the rotation number changes by $\pm(\rot_f(\partial_j\Sigma) + 1) 
= \pm\nu_j$. Hence there exists a non-separating simple closed curve
$\gamma'$ in $\Sigma'$ such that $\rot_f(\gamma) - \rot_f(\gamma') 
\in \mathfrak{b}(f)$. The curve $\gamma'$ is mapped to $\alpha$ 
by an element of the subgroup generated by the Dehn twists $t_\alpha$ 
and $t_\beta$. For any simple closed curve $\gamma''$ in $\Sigma$, 
we have 
\begin{equation}\label{eq:tbeta}
\rot_f(t_\beta(\gamma'')) - \rot_f(\gamma'') = 
([\gamma'']\cdot[\beta])\rot_f(\beta) \in \mathfrak{b}(f)
\end{equation}
and 
$\rot_f(t_\alpha(\gamma'')) - \rot_f(\gamma'') \in \mathfrak{b}(f)$.
Hence we have $\rot_f(\gamma') \in \rot_f(\alpha) + \mathfrak{b}(f)
= \mathfrak{b}(f)$. This proves $\mathfrak{a}(f) \subset \mathfrak{b}(f)$, 
and completes the proof of the lemma.
\end{proof}
\begin{cor}\label{cor:Arf}
If $\rot_f(\partial_j\Sigma)$ is odd for any $0 \leq j\leq n$, we have
$$
\operatorname{Arf}(\xi_2(f)) \equiv \tilde A(f) +1 \pmod{2}.
$$
\end{cor}
\begin{proof} By Lemma \ref{lem:omegarot} we have
$
\operatorname{Arf}(\xi_2(f)) \equiv (\rot_f(\alpha)+1)(\rot_f(\beta)+1) \pmod{2}
$.
\end{proof}

\begin{thm}\label{genusone}
Suppose $g = 1$, and $f_1, f_2 \in F(\Sigma_{1,n+1})$.
Then $f_1$ and $f_2$ belong to the same 
$\mathcal{M}(\Sigma_{1,n+1})$-orbit, if and only if
$f_1$ and $f_2$ satisfy both of the following conditions 
\begin{enumerate}
\item[(i)] $
\rot_{f_1}(\partial_j\Sigma) = \rot_{f_2}(\partial_j\Sigma)
$
for any $0 \leq j \leq n$.
\item[(ii)] $\tilde A(f_1) = \tilde A(f_2)\in \mathbb{Z}_{\geq0}$.
\end{enumerate}
\end{thm}
\begin{proof} If $f_1$ and $f_2$ belong to the same 
$\mathcal{M}(\Sigma)$-orbit, then it is clear that 
they satisfy both of the conditions. Hence it suffices to prove the following:
For any $f \in F(\Sigma)$ we have $(\rot_{f\varphi}(\alpha), 
\rot_{f\varphi}(\beta)) = (\tilde A(f), 0) \in \Z^2$ for some $\varphi
\in \mathcal{M}(\Sigma)$. \par
From the formula (\ref{eq:tbeta}) and the similar one for $t_\alpha$, 
the actions of $t_\alpha$ and $t_\beta$ on the row vectors 
$(\rot_f(\alpha), \rot_f(\beta)) \in \Z^2$ generate the standard right 
action of $SL_2(\Z)$ on $\Z^2$. By the Euclidean algorithm,
the vectors $(a_1, b_1)$ and $(a_2, b_2) \in \Z^2$ belong to the same
$SL_2(\Z)$-orbit if and only if $\gcd(a_1, b_1)=\gcd(a_2, b_2) \in \Z$.
\par
We denote $d:= \gcd(\rot_f(\alpha), \rot_f(\beta))$ and 
$c:= \gcd(\nu_j(f);\,\, 0 \leq j \leq n)$. 
Then $\tilde A(f) = \gcd(c,d)$. 
By the Euclidean algorithm, we have 
$(\rot_{f\varphi_1}(\alpha), \rot_{f\varphi_1}(\beta)) = (d,0)$
for some $\varphi_1 \in \mathcal{M}(\Sigma)$.
Recall the non-separating simple closed curve $\alpha^{(j)}$ 
introduced in the proof of Lemma \ref{lem:bf}. 
For any $f' \in F(\Sigma)$ we have 
$$
\aligned
& \rot_{f'}({t_\alpha}^{-1}t_{\alpha^{(j)}}(\alpha)) = 
\rot_{f'}(\alpha), \quad\text{and}\\
& \rot_{f'}({t_\alpha}^{-1}t_{\alpha^{(j)}}(\beta)) = 
\rot_{f'}(t_{\alpha^{(j)}}(\beta)) + ([\beta]\cdot[\alpha])\rot_{f'}(\alpha)\\
& = \rot_{f'}(\beta) - ([\beta]\cdot[\alpha])\rot_{f'}(\alpha^{(j)})
+ ([\beta]\cdot[\alpha])\rot_{f'}(\alpha)\\
& = \rot_{f'}(\beta) - \nu_j(f').
\endaligned
$$
Hence there exists an element $\varphi_2$ in the subgroup 
generated by the elements ${t_\alpha}^{-1}t_{\alpha^{(j)}}$, $0 \leq j \leq n$, 
such that $(\rot_{f\varphi_1\varphi_2}(\alpha), 
\rot_{f\varphi_1\varphi_2}(\beta)) = (d,c)$. 
Recall $\tilde A(f) = \gcd(c,d)$. 
By the Euclidean algorithm, we have 
$(\rot_{f\varphi_1\varphi_2\varphi_3}(\alpha), 
\rot_{f\varphi_1\varphi_2\varphi_3}(\beta)) = (\tilde A(f),0)$
for some $\varphi_3 \in \mathcal{M}(\Sigma)$.
This proves the theorem.
\end{proof}
\begin{cor}\label{cor:genus1}
For $\nu = (\nu_j)^n_{j=0} \in \Z^{n+1}\setminus\{0\}$ with 
$\sum^n_{j=0}\nu_j= 0$, the inverse image $\rho^{-1}(\nu)$ 
is parametrized by the positive divisors of $\gcd(\nu_j; 
0 \leq j \leq n)$, while $\rho^{-1}(0)$ by the non-negative 
integers $\Z_{\geq0}$.
\end{cor}
\begin{proof} If $\nu\neq0$, then $\tilde A(f)$ is a positive 
divisor of the $\gcd$. The corollary follows from 
Lemma \ref{lem:betti}. See also the equation 
(\ref{eq:rimage}).
\end{proof}

The following is related to the formality problem of the Turaev cobracket
on genus $1$ surfaces \cite{AKKN}.
\begin{cor}\label{cor:realize1}
For $f \in F(\Sigma_{1,n+1})$, there exists a mapping class $\varphi
\in \mathcal{M}(\Sigma_{1,n+1})$ satisfying 
$\rot_{f\varphi}(\alpha) = \rot_{f\varphi}(\beta) = 0$, 
if and only if $\tilde A(f) = \gcd(\nu_j; 0 \leq j \leq n)$.
\end{cor}
\begin{proof}
By Lemma \ref{lem:betti}, there exists a framing $f_\bullet \in F(\Sigma_{1,n+1})$
such that $\rot_{f_\bullet}(\alpha) = \rot_{f_\bullet}(\beta) = 0$ and
$\nu_j(f_\bullet) = \nu_j(f)$ for any $0 \leq j \leq n$.
Then $\tilde A(f_\bullet) = \gcd(\nu_j; 0 \leq j \leq n)$ from Lemma \ref{lem:bf}.
Hence the corollary follows from Theorem \ref{genusone}.
\end{proof}

\subsection{The relative genus $1$ case}

We conclude this paper by some discussion about the relative version
\cite{RW14}, which we will need to describe to the self-intersection
of an immersed path.  
Here we fix a framing of the tangent bundle restricted to 
the boundary $\delta: T\Sigma\vert_{\partial\Sigma}\overset\cong\to
\partial\Sigma \times\mathbb{R}^2$, and consider the set $F(\Sigma, \delta)$
of homotopy classes of framings $f: T\Sigma\overset\cong\to 
\Sigma\times\mathbb{R}^2$ which extend the framing $\delta$, 
where all the homotopies we consider fix $\delta$ pointwise.
By Lemma \ref{lem:PH} and some obstruction theory, the set 
$F(\Sigma, \delta)$ is not empty if and only if $\sum^n_{j=0} \rot_\delta(
\partial_j\Sigma) = \chi(\Sigma)$. For the rest, we assume 
$F(\Sigma, \delta) \neq \emptyset$. 
In this setting, for any $f \in F(\Sigma, \delta)$, 
we can consider the rotation number $\rot_f(\ell) \in \mathbb{R}$
of an immersed path $\ell$ connecting two different points 
on the boundary $\partial \Sigma$. 
We denote by $1 \in S^1$ the unit element of $S^1 = SO(2)$. 
The group $[(\Sigma, \partial\Sigma), (S^1, 1)] = H^1(\Sigma, \partial\Sigma; \mathbb{Z}) 
= H^1(\Sigma, \partial\Sigma)$ acts on the set $F(\Sigma, \delta)$ freely and 
transitively. For $1 \leq j \leq n$, we choose a point $*_j \in \partial_j\Sigma$ 
and a simple arc $\eta_j$ from a point on $\partial_0\Sigma$ to $*_j$
such that each $\eta_j$ is disjoint from $\{\alpha_i, \beta_i\}^g_{i=1} \cup
\{\eta_k\}_{k\neq j}$, and transverse to $\partial_0\Sigma$ and
$\partial_j\Sigma$. Then the homology classes 
$\{[\alpha_i], [\beta_i]\}^g_{i=1} \cup
\{[\eta_j]\}^n_{j=1}$ constitute a free basis of $H_1(\Sigma, \partial\Sigma)$. 
The evaluation map
\begin{equation}
\operatorname{Ev}: F(\Sigma, \delta) \to \mathbb{Z}^{2g+n}, \quad 
f \mapsto ((\rot_f(\alpha_i), \rot_f(\beta_i))^g_{i=1}, 
(\lceil\rot_f(\eta_j)\rceil)^n_{j=1})
\label{eq:relbij}
\end{equation}
is bijective, and compatible with the action of $H^1(\Sigma, \partial\Sigma)$. 
Here $\lceil\rot_f(\eta_j)\rceil \in \mathbb{Z}$ is the ceiling of the rotation number
$\rot_f(\eta_j) \in \mathbb{R}$. Randal-Williams \cite{RW14} introduced 
the generalized Arf invariant $\gArf(f) \in \mathbb{Z}/2$ by 
\begin{equation}\label{eq:gArf}
\gArf(f) := \sum^g_{i=1}(\rot_f(\alpha_i)+1)(\rot_f(\beta_i)+1)
+ \sum^n_{j=1} \nu_j\lceil\rot_f(\eta_j)\rceil \bmod {2} \in \mathbb{Z}/2,
\end{equation}
which is denoted by $A(f)$ in the original paper \cite{RW14}.
The mapping class group $\mathcal{M}(\Sigma)$ acts on the set $F(\Sigma, \delta)$
in a natural way. As was proved in \cite{RW14}, the generalized 
Arf invariant is invariant under the mapping class group action for any $g \geq 0$, 
and, if $g \geq 2$, 
the orbit set $F(\Sigma, \delta)/\mathcal{M}(\Sigma)$ is of cardinality $2$ or $0$ 
for any $\delta$, and described by the generalized Arf invariant.\par

Now we consider the case $g=1$. We use the notation in \S2.3.
The invariant $\tilde A(f)$ is related to the generalized Arf invariant
$\gArf(f)$ as follows.\par
\begin{enumerate}
\item[(1)] Suppose $\tilde A(f)$ is even. Then $\rot_f(\alpha)$, $\rot_f(\beta)$ 
and all of $\nu_j$'s are even. Hence $\gArf(f) \equiv (\rot_f(\alpha)+1)
(\rot_f(\beta)+1) \equiv 1 \bmod 2$. If $f_1 \in F(\Sigma, \delta)$ is given 
by $\operatorname{Ev}(f_1) = ((\tilde A(f), 0), (0, \dots, 0))$, then we have 
$\tilde A(f_1) = \tilde A(f)$.
\item[(2)] Next we consider the case $\tilde A(f)$ is odd and $\gArf(f) = 0 \bmod 2$. 
If $f_2 \in F(\Sigma, \delta)$ is given 
by $\operatorname{Ev}(f_2) = ((\tilde A(f), 0), (0, \dots, 0))$, then we have 
$\tilde A(f_2) = \tilde A(f)$ and $\gArf(f_2) = 0 \bmod 2$.
\item[(3)] Finally assume $\tilde A(f)$ is even and $\gArf(f) = 1 \bmod 2$. 
Then we have $\nu_j \equiv 1\pmod{2}$ for some $1 \leq j \leq n$.
If not, $\rot_f(\alpha)$ or $\rot_f(\beta)$ are odd, so that 
$\gArf(f) = 0 \bmod 2$. This contradicts the assumption. 
Let $j_0$ be the maximum $j$ satisfying $\nu_j \equiv 1\pmod{2}$.
If $f_3 \in F(\Sigma, \delta)$ is given 
by $\operatorname{Ev}(f_3) = ((\tilde A(f), 0), (0, \dots, 0, \overset{j_0}{\breve{1}}, 0, \dots, 0))$, then we have 
$\tilde A(f_3) = \tilde A(f)$ and $\gArf(f_3) = 1 \bmod 2$.
\end{enumerate}
From Lemma \ref{lem:bf} the invariant $\tilde A(f)$ can be realized to be any non-negative 
divisor of $\gcd(\nu_j; 0 \leq j \leq n)$.
Here we agree that any integer is a divisor of $0$.

\begin{thm}\label{relzero} Suppose $g = 1$ and 
$F(\Sigma, \delta) \neq\emptyset$.
Then the orbit set 
$F(\Sigma, \delta)/\mathcal{M}(\Sigma)$ is parametrized by 
the invariant $\tilde A(f)$ and the generalized Arf invariant 
$\gArf(f)$. More precisely, for any $f \in F(\Sigma)$, we have 
$f = f_k\circ\varphi$ for some $\varphi \in \mathcal{M}(\Sigma)$ 
and $k= 1,2,3$. Here we choose $f_k$ according to 
the invariants $\tilde A(f)$ and $\gArf(f)$ as stated above.
\end{thm}
\begin{proof}
We may assume each $\eta_j$ is disjoint from the subsurface $\Sigma'
(\cong \Sigma_{1,1})$, a regular neighborhood of $\alpha(S^1)\cup\beta(S^1)$.
There is an element $\tau \in \mathcal{M}(\Sigma)$ whose support is in $\Sigma'$
such that $(\rot_{f\circ\tau}(\alpha), \rot_{f\circ\tau}(\beta)) = 
(-\rot_{f}(\alpha), -\rot_{f}(\beta))$ for any $f \in F(\Sigma, \delta)$. In fact, 
$\tau$ can be obtained as some product of Dehn twists $t_\alpha$ and $t_\beta$.
In particular, we have $\rot_{f\circ\tau}(\eta_j) = \rot_f(\eta_j)$ for any 
$0 \leq j \leq n$. \par
Next we consider a framing $f \in F(\Sigma, \delta)$ which satisfies 
$\operatorname{Ev}(f) = ((A, 0), (\rho_1, \dots, \rho_n))$ 
for some $\rho_j \in \mathbb{Z}$. Here we assume $A = \tilde A(f)$. 
We remark that $A$ divides any $\nu_j$, $0 \leq j \leq n$. 
Recall the non-separating simple closed curve $\alpha^{(j)}$ 
introduced in the proof of Lemma \ref{lem:bf}. 
Here we choose the band connecting $\alpha$ and $\partial_j\Sigma$ 
to be disjoint from any $\eta_k$, $1 \leq k \leq n$. Then, the curve 
$\alpha^{(j)}$ is disjoint from $\eta_k$ for $k \neq j$, and we may assume
$\alpha^{(j)}$ and $\eta_j$ intersect transversely to each other at the unique point.
We define $\psi_j := t_{\alpha^{(j)}}{t_\alpha}^{-1-(\nu_j/A)}
{t_{\partial_j\Sigma}}^{-1} \in \mathcal{M}(\Sigma)$. Since $\rot_f(\alpha^{(j)})
= A + \nu_j$ and $\rot_f(\partial_j\Sigma) = \nu_j-1$, we have 
$\rot_{f\circ\psi_j}(\eta_j) - \rot_f(\eta_j) = -A -\nu_j + \nu_j -1 = -A -1$ and
$\rot_{f\circ\psi_j}(\beta) - \rot_f(\beta) = A +\nu_j -(1+(\nu_j/A))A = 0$. 
Clearly we have 
$\rot_{f\circ\psi_j}(\alpha) = \rot_f(\alpha) = A$ and
$\rot_{f\circ\psi_j}(\eta_k) - \rot_f(\eta_k) = 0$ for $k\neq j$.
Moreover we define $\psi'_j := \tau t_{\alpha^{(j)}}{t_\alpha}^{-1+(\nu_j/A)}
{t_{\partial_j\Sigma}}^{-1}\tau^{-1} \in \mathcal{M}(\Sigma)$. 
Similarly we have 
$\rot_{f\circ\psi'_j}(\eta_j) - \rot_f(\eta_j) = A-1$, 
$\rot_{f\circ\psi'_j}(\alpha) = \rot_f(\alpha) = A$, 
$\rot_{f\circ\psi'_j}(\beta) = \rot_f(\beta)$ and
$\rot_{f\circ\psi'_j}(\eta_k) - \rot_f(\eta_k) = 0$ for $k\neq j$.
As a consequence of the construction of $\psi_j$ and $\psi'_j$, 
there is some $\varphi_j \in \mathcal{M}(\Sigma)$ and $\epsilon_j \in \{0,1\}$ 
such that $\operatorname{Ev}(f\circ\varphi_j) = ((A, 0),(\rho_1,\dots, \rho_{j-1}, 
\epsilon_j, \rho_{j+1}, \dots, \rho_n))$ and $\epsilon_j \equiv \rho_j \pmod{2}$.
In fact, $\gcd\{-A-1, A-1\}$ divides $2$.\par
Now we consider an arbitrary element $f_0 \in F(\Sigma, \delta)$. 
We denote $A = \tilde A(f_0)$. 
From the proof of Theorem \ref{genusone}, we have 
$\operatorname{Ev}(f_0\circ\varphi_0) = ((A, 0), (\rho^0_1, \dots, \rho^0_n))$ for some $\varphi_0\in \mathcal{M}(\Sigma)$ and $\rho^0_j \in \mathbb{Z}$. \par
(1) Suppose $A = \tilde A(f)$ is even. 
Then we may assume each $\rho^0_j$ is even. 
In fact, $\rot_{f\circ t_{\partial_j\Sigma}}(\eta_j) 
- \rot_f(\eta_j) = -\rot_f(\partial_j\Sigma) =  - \nu_j + 1$ is odd for any 
$f \in F(\Sigma, \delta)$. 
Hence we have some suitable product $\tilde\varphi \in \mathcal{M}(\Sigma)$ 
of $\varphi_j \in \mathcal{M}(\Sigma)$'s stated above such that
$\operatorname{Ev}(f_0\circ\varphi_0\circ\tilde\varphi) = ((A, 0), (0, \dots, 0))$.
This means $f_0\circ\varphi_0\circ\tilde\varphi = f_1 \in F(\Sigma, \delta)$,
as was desired.\par
(2) Assume $A = \tilde A(f)$ is odd and $\gArf(f) = 0\bmod2$.
Then we have $0 = \gArf(f_0) = \gArf(f_0\circ\varphi_0) \equiv 
A + 1 + \sum^n_{j=1}\nu_j\lceil \rot_{f_0\circ\varphi_0}(\eta_j)\rceil
\equiv \sum^n_{j=1}\nu_j\lceil \rot_{f_0\circ\varphi_0}(\eta_j)\rceil
\pmod{2}$.
Hence there are some $1 \leq j_1 < j_2 < \cdots < j_{2m} \leq n$
such that $\nu_{j_s}\lceil \rot_{f_0\circ\varphi_0}(\eta_{j_s})\rceil 
\equiv 1 \pmod{2}$ and 
$\nu_{j}\lceil \rot_{f_0\circ\varphi_0}(\eta_{j})\rceil 
\equiv 0 \pmod{2}$ if $j \not\in \{j_1, j_2, \dots, j_{2m}\}$. 
We choose a band connecting $\partial_{j_1}(\Sigma)$ and 
$\partial_{j_2}(\Sigma)$ disjoint from $\alpha$, $\beta$ and $\eta_k$ for 
$k\neq j_1, j_2$ to obtain a separating simple closed curve $\lambda$ such that
$\partial_{j_1}(\Sigma)$, $\partial_{j_2}(\Sigma)$ and $\lambda$ bound a pair of 
pants. Then $\rot_{f_0\circ\varphi_0}(\lambda) = \nu_{j_1} + \nu_{j_2} - 1$ 
is odd. Hence we have 
$\lceil \rot_{f_0\circ\varphi_0\circ t_\lambda}(\eta_{j_1})\rceil \equiv 
\lceil \rot_{f_0\circ\varphi_0\circ t_\lambda}(\eta_{j_2})\rceil \equiv 0 
\pmod{2}$. By similar consideration we obtain 
some $\varphi' \in \mathcal{M}(\Sigma)$ such that 
$\lceil \rot_{f_0\circ\varphi_0\circ \varphi'}(\eta_{j})\rceil \equiv 0 
\pmod{2}$ for any $1 \leq j \leq n$.
Hence we have some suitable product $\tilde\varphi \in \mathcal{M}(\Sigma)$ 
of $\varphi_j \in \mathcal{M}(\Sigma)$'s such that
$\operatorname{Ev}(f_0\circ\varphi_0\circ \varphi'\circ\tilde\varphi) 
= ((A, 0), (0, \dots, 0))$.
This means $f_0\circ\varphi_0\circ \varphi'\circ\tilde\varphi = f_2
\in F(\Sigma, \delta)$,
as was desired.\par
(3)  Assume $A = \tilde A(f)$ is odd and $\gArf(f) = 1\bmod2$.
Then $\sum^n_{j=1}\nu_j\lceil \rot_{f_0\circ\varphi_0}(\eta_j)\rceil
\equiv 1 \pmod{2}$.
Hence there are some $1 \leq j_1 < j_2 < \cdots < j_{2m-1} \leq n$
such that $\nu_{j_s}\lceil \rot_{f_0\circ\varphi_0}(\eta_{j_s})\rceil 
\equiv 1 \pmod{2}$ and 
$\nu_{j}\lceil \rot_{f_0\circ\varphi_0}(\eta_{j})\rceil 
\equiv 0 \pmod{2}$ if $j \not\in \{j_1, j_2, \dots, j_{2m-1}\}$. 
In a similar way to (2), we obtain 
some $\varphi' \in \mathcal{M}(\Sigma)$ such that 
$\lceil \rot_{f_0\circ\varphi_0\circ \varphi'}(\eta_{j_0})\rceil \equiv 1 
\pmod{2}$ and 
$\lceil \rot_{f_0\circ\varphi_0\circ \varphi'}(\eta_{j})\rceil \equiv 0 
\pmod{2}$ for any $j \neq j_0$.
Hence we have some suitable product $\tilde\varphi \in \mathcal{M}(\Sigma)$ 
of $\varphi_j \in \mathcal{M}(\Sigma)$'s such that
$\operatorname{Ev}(f_0\circ\varphi_0\circ \varphi'\circ\tilde\varphi) 
= ((A, 0), (0, \dots, 0, \overset{j_0}{\breve{1}}, 0, \dots, 0))$.
This means $f_0\circ\varphi_0\circ \varphi'\circ\tilde\varphi = f_3
\in F(\Sigma, \delta)$,
as was desired.\par
This completes the proof of the theorem.
\end{proof}

The situation for the relative genus $0$ case is elementary, but 
seems too complicated to describe by some simple invariants.

\vskip 10mm
\noindent
Department of Mathematical Sciences, \\
University of Tokyo \\
3-8-1 Komaba, Meguro-ku, Tokyo, \\
153-8914, JAPAN. \\
kawazumi@ms.u-tokyo.ac.jp\\
\\
\end{document}